\documentclass[article,12pt]{amsart}
\usepackage{amscd}
\usepackage{bbm}
\usepackage{dsfont}
\usepackage{datetime}
\usepackage{latexsym}
\usepackage{srcltx}
\usepackage{amsfonts}
\usepackage{amssymb}
\usepackage{setspace}

\input epsf
\newtheorem{theorem}{Theorem}
\newtheorem{proposition}[theorem]{Proposition}

\newtheorem{problem} [theorem] {Problem}
\newtheorem{lemma}[theorem]{Lemma}

\theoremstyle{definition}
\newtheorem{example}[theorem]{Example}
\newtheorem{definition}[theorem]{Definition}

\newtheorem{remark} [theorem] {Remark}

\begin{document}
\title{ Strong Monotonicity of Spectral Radius of Positive Operators.}
\author{ Donald W. Hadwin}
\address{ University of New Hampshire, 105 Main Street, Durham, NH 03824}
\email{don@unh.edu}
\author{ Arkady K. Kitover}
\address{ Community College of Philadelphia, 1700 Spring Garden Street,
Philadelphia, PA 19130}
\email{akitover@ccp.edu}
\author{Mehmet Orhon}
\address{ University of New Hampshire, 105 Main Street, Durham, NH 03824}
\email{mo@unh.edu}
\date{\today }
\maketitle

\begin{abstract}
The classic result of Perron and Frobenius states that if $A$ and $B$ are
matrices with nonnegative elements, such that $A \leq B$, $A$ is
irreducible, and $\rho(A) = \rho(B)$ then $A = B$. We extend this result to
a large class of band irreducible positive operators on a large class of
Banach lattices and provide examples to show that the conditions we put on
the operators and the Banach lattices cannot be weakened.
\end{abstract}

A part of the famous Perron - Frobenius theorem states that if $A$ and $B$
are matrices with nonnegative elements, such that $A\leq B$, $A$ is
irreducible, and $\rho (A)=\rho (B)$ where $\rho (A)$ means the spectral
radius of $A$ then $A=B$. It was proved by Perron in \cite{Pe} in 1907 for
matrices with strictly positive elements and by Frobenius in \cite{Fr} in
1912 for arbitrary irreducible matrices. The proof can be also found in \cite%
[Section 2.1]{Va}. While the part of Perron - Frobenius theorem which
describes the structure of peripheral spectrum of nonnegative irreducible
matrices became a subject of very active research and was successively
extended to include the class of band irreducible positive compact operators
on Banach lattices (an extensive bibliography can be found in \cite{Gr} and
in~\cite{AA}) the above statement about spectral radii until recently
attracted comparatively less attention.

Nevertheless in 1970 Marek~\cite[Theorem 4.3]{Ma} proved the following very
general result.

\begin{theorem}
\label{t2} (Marek~\cite{Ma}) Let $X$ be a partially ordered Banach space
with the closed normal $B$-cone $K$. Let $T_1$, $T$ be positive linear
operators on $X$, $T_1 \leq T$, and $T$ be semi non-supporting. Let $\rho(T)$
be an eigenvalue of $T$ and the Banach conjugate operator $T^\prime$ to
which correspond a positive eigenvector and respectively a positive
eigenfunctional. Let also $\rho(T_1)$ be an eigenvalue of $T_1^\prime$ with
a positive eigenfunctional. Finally assume that $\rho(T) = \rho(T_1)$. Then $%
T = T_1$.
\end{theorem}

Let us discuss the meaning of conditions of Theorem~\ref{t2} in the case
when the partially ordered Banach space $X$ is a Banach lattice.

The condition that the cone $K$ is a $B$-cone was introduced by Sawashima in~%
\cite{Sa} and means the following:

\begin{equation*}
\exists C >0, \forall x \in X \exists \{x_{1n}\}_{n \in \mathds{N}},
\{x_{2n}\}_{n \in \mathds{N}} \subset K \; \mathrm{such \; that}\; x_{1n} -
x_{2n} \mathop \to \limits_{n \to \infty} x \;
\end{equation*}
\begin{equation*}
\mathrm{and}\; \|x_{in}\| \leq C\|x\|, i = 1,2, n \in \mathds{N}.
\end{equation*}
This condition is obviously satisfied for the cone of positive elements in a
Banach lattice.

The condition that an operator $T$ is semi non-supporting was also
introduced in~\cite{Sa} and means that for any nonzero $x \in K$ and $%
x^\prime$ in $K^\prime$ there is a $p = p(x,x^\prime) \in \mathds{N}$ such
that $(T^p x, x^\prime) > 0$. In the case of Banach lattices this condition
is well known to be equivalent to $T$ being ideal irreducible. For the sake
of completeness we provide a short proof.

\begin{proposition}
\label{p1} Let $T$ be a positive operator on a Banach lattice $X$. Then $T$
is semi non-supporting if and only if it is ideal irreducible.
\end{proposition}

\begin{proof}
Assume first that there is a nontrivial closed ideal $J$ in $X$ such that $%
TJ \subseteq J$. Then the annihilator of $J$ is a nontrivial band in $%
X^\prime$ and therefore there is a positive nonzero $x^\prime \in X^\prime$
such that $x^\prime(J) = 0$. Let $x \in J$ be positive and nonzero. Then $%
(T^px,x^\prime)=0, p \in \mathds{N}$ whence $T$ is not semi non-supporting.

Next assume that there are positive nonzero $x \in X$ and $x^\prime \in
X^\prime$ such that $(T^px,x^\prime)=0, p \in \mathds{N}$. Let $z = \sum
\frac{T^nx}{2^n \|T^nx\|}$ where the sum is taken over all $n \in \mathds{N}$
such that $T^nx \neq 0$. Then the closed principal ideal generated by $z$ is
$T$-invariant and nontrivial.
\end{proof}

Thus Theorem~\ref{t2} provides a very satisfactory extension of the
corresponding part of the Perron - Frobenius theorem for the ideal
irreducible operators. Our goal in this paper is to obtain a similar
extension for the larger class of \textbf{band irreducible} operators.

In fact, in the case when the Banach lattice $X$ has order-continuous norm
every closed ideal is a band and there is nothing to prove. But, as the next
example shows, for an arbitrary Banach lattice we cannot just substitute
band irreducible for ideal irreducible in the conditions of Theorem~\ref{t2}.

\begin{example}
\label{e1} Let $X=\{x\in C[0,1]\;:\;x(0)=x(1)\}$. Clearly $X$ is a Banach
lattice. Let $\alpha $ be an irrational number from $(0,1)$. By $t\dotplus
\alpha $ we will mean addition modulo $1$. Let $l\in X$ such that $l(1/2)=0$%
, $l(t)>0$, if $t\in \lbrack 0,1]$ and $t\neq 1/2$, and $\int\limits_{0}^{1}%
\ln {l(t)}dt=-\infty $. For example we can take
\begin{equation*}
l(t)=%
\begin{cases}
e^{-\frac{1}{(t-1/2)^{2}}}, & \text{if $t\neq 1/2$,} \\
0, & \text{ if $t=1/2$}%
\end{cases}%
\end{equation*}

We introduce the following positive operators.

\begin{equation*}
(Ax)(t) = x(1/2), x \in X, t \in [0,1],
\end{equation*}
\begin{equation*}
(Bx)(t) = l(t)x(t \dotplus \alpha), x \in X, t \in [0,1],
\end{equation*}
\begin{equation*}
C = A+B.
\end{equation*}

It is obvious that the operator $A$ is one dimensional and that its spectrum
$\sigma(A) = \{0,1\}$. The operator $B$ is band irreducible because $\alpha$
is an irrational number. Moreover the condition $\int \limits_0^1 \ln{l(t)}%
dt = - \infty$ guarantees that $B$ is quasinilpotent (see e.g., ~\cite{Da}).
Next notice that $\rho(C) \geq \rho(A) = 1$ and that because $C$ is a
positive operator $\rho(C) \in \sigma(C)$. The operator $C$ is a compact
perturbation of $B$. Clearly the Fredholm \footnote{%
As usual we understand by Fredholm spectrum of a bounded operator on a
Banach space its spectrum in the Calkin algebra.} spectrum of $B$, $%
\sigma_f(B)$, is the singleton $\{0\}$. The Fredholm spectrum does not
change under compact perturbations (see~\cite[Theorem 5.35, p.244]{Ka})
whence $\sigma_f(C) = \{0\}$ and by~\cite[Theorem 5.33, p. 243]{Ka} $\rho(C)$
is an isolated point and an eigenvalue of $C$ of finite multiplicity. We
claim that $\rho(C) = 1$. Indeed, assume to the contrary that $\rho(C) = r >
1$. Let $z \in X$ be the corresponding eigenvector. Then
\begin{equation*}
l(t)z(t \dotplus \alpha) + z(1/2) = rz(t), \; t \in [0,1]
\end{equation*}
Because $l(1/2) = 0$ we see that $z(1/2) = 0$ whence $l(t)z(t \dotplus
\alpha) = rz(t), \; t \in [0,1]$ in contradiction to $B$ being
quasinilpotent.

We claim additionally that the spectral projection corresponding to the
isolated eigenvalue 1 is one dimensional. Indeed, assume first that there
are two linearly independent eigenvectors $f$ and $g$ of $C$ corresponding
to the eigenvalue 1. Then $f(1/2) \neq 0$ and $g(1/2) \neq 0$. Therefore,
without loss of generality we can assume that $f(1/2) = g(1/2)$. But then $%
B(f-g) = f - g$ whence $f = g$, a contradiction.

Assume now that $f$ is an eigenvector of $C$ corresponding to the eigenvalue
1, and that $h$ is a root vector of $C$, i.e. $(I - C)h = f$. Then
\begin{equation*}
l(t)h(t \dotplus \alpha) + h(1/2) - h(t) = f(t), t \in [0,1].
\end{equation*}
If we put $t=1/2$ in the above equation we get $f(1/2) = 0$ whence $Bf =f$,
a contradiction again.

Now we see that there is a positive eigenvector of $C$ corresponding to the
eigenvalue 1. Indeed, because the spectral projection $P$ on the eigenspace
of $C$ corresponding to the eigenvalue 1 is one dimensional we can apply the
well known formula
\begin{equation*}
P = \lim \limits_{r \to 1+} (r-1)\sum \limits_{n=0}^\infty \frac{C^n}{r^{n+1}%
} \eqno(\bigstar)
\end{equation*}
(see e.g., ~\cite[p. 329]{Sc}), that shows that $P$ is a positive operator.

Similar reasoning applied to the Banach dual operators $A^\prime$, $B^\prime$%
, and $C^\prime$ shows that $\rho(C^\prime) =1$ and $1$ is an eigenvalue of $%
C^\prime$ to which corresponds a positive eigenfunctional ( notice that the
spectral projection on the eigenspace of $C^\prime$ corresponding to the
eigenvalue 1 coincides with $P^\prime$ and therefore is one dimensional).

Thus we see that the operators $A$ and $C$ satisfy all the conditions of
Theorem~\ref{t2} with the exception that instead of being ideal irreducible $%
C$ is only band irreducible (clearly $C$ leaves invariant the ideal $I = \{x
\in X : x(1/2) = 0 \}$), but exactly because of this exception the statement
of the said theorem is not true anymore.
\end{example}

A quick look at Example~\ref{e1} allows us to notice the reason for this
failure; the operators $A$ and $C$ are not order continuous. Indeed, Alekhno
(see~\cite[Theorem 5]{Al}) proved that the conclusion of Theorem~\ref{t2}
remains true under the assumption that the Banach lattice $X$ has a
separating set of order continuous functionals and the dominating operator $%
T_1$ is band irreducible and $\sigma$-order continuous (see Definition~\ref%
{d1}).

The main result of the current paper (Theorem~\ref{t1}, part (a)) extends
Alekhno's result to the much larger class of Banach lattices having the
property that the sequential Lorentz seminorm (see Definition~\ref{d2}) is a
norm. Example~\ref{e2} shows that this is as far as we can go.

We need to recall two definitions.

\begin{definition}
\label{d1} (See e.g., \cite[Page 19]{AA}). A positive operator $T$ on a
Banach lattice $X$ is called $\sigma$-order continuous if for every sequence
$\{x_n \} \subset X$ such that $x_n \downarrow 0$ we have $Tx_n \downarrow 0$%
.
\end{definition}

\begin{remark}
\label{r4} There are different definitions of $\sigma$-order continuity for
operators on vector lattices but in the case of positive operators they all
are equivalent to Definition~\ref{d1}. We refer the reader to~\cite{KW} for
a detailed discussion of this question.
\end{remark}

\begin{definition}
\label{d2} Let $X$ be a Banach lattice. The sequential Lorentz seminorm $l$
on $X$ is defined as
\begin{equation*}
l(x) = \inf \limits_{x_n \uparrow |x|} \lim \limits_n \|x_n\|, x \in X %
\eqno{(\star)}
\end{equation*}
where $\inf$ is taken over all the sequences $\{x_n \} \subset X$ such that $%
x_n \geq 0$ and $x_n \uparrow |x|$.
\end{definition}

\begin{remark}
\label{r1} The class of Banach lattices $X$ for which the sequential Lorentz
seminorm is actually a lattice norm on $X$ is very large. For example it
includes all the Banach lattices with the weak sequential Fatou property
(i.e., there is a constant $K \geq 1$ such that $x_n \uparrow x \Rightarrow
\|x\| \leq K \sup \limits_n \|x_n\|$ (see \cite[p.89]{AA1})) as well as all
Banach Lattices with a separating set of $\sigma$-order continuous
functionals. Important examples are provided by the Banach spaces of
measurable functions.
\end{remark}

We will need the following lemma.

\begin{lemma}
\label{l3} Let $X$ be a Banach lattice such that $dim X >1$. Let $T: X \to X$
be a positive, band-irreducible, and $\sigma$-order continuous operator such
that its spectral radius $\rho(T)$ is a pole of the resolvent $R(\lambda,T)$.

Assume additionally at least one of the following two conditions:

\textrm{(a)} The sequential Lorentz semi-norm $l$ is a norm on $X$;

\textrm{(b)} The operator $T$ is weakly compact.

\noindent Then $\rho(T) > 0$, the multiplicity of the pole $\rho(T)$ is $1$,
the corresponding eigenspace is one dimensional, and the spectral projection
$P_T$ is $\sigma$-order continuous.
\end{lemma}

\begin{proof}
The inequality $\rho(T) >0$ follows trivially from the fact that $T$ is band
irreducible (see~\cite{Ki}).

It was proved in~\cite[Theorem 6 and Lemmas 7 and 8]{Ki} that the conditions
in part $(a)$ imply the conclusion of the lemma.

To prove the conclusion of the lemma in case part (b) holds, assume that $%
\rho(T) = 1$. Let $p$ be the order of the pole at $1$ and let $Q_{-p} = \lim
\limits_{\lambda \downarrow 1} (\lambda -1)^pR(\lambda,T)$ where the limit
is in operator norm. Notice that the operator $Q_{-p}$ is positive and
commutes with $T$. If we prove that $Q_{-p}$ is $\sigma$-order continuous
then we can repeat arguments from the proof of Theorem 6 in~\cite{Ki} and
finish the proof. Assume to the contrary that there is a sequence $x_n \in
X, \; n \in \mathds{N}$ such that $x_n \downarrow 0$ but $Q_{-p} x_n \geq y
\gvertneqq 0$. Notice that $Tx_n \downarrow 0$ and $T$ is weakly compact
whence $Tx_n$ weakly converges to $0$. Therefore $Q_{-p}Tx_n$ weakly
converges to $0$. But $Q_{-p}Tx_n =TQ_{-p}x_n \geq Ty$ whence $Ty = 0$.
Therefore $T\{y\}^{dd} = 0$ in contradiction to $T$ being band irreducible.
\end{proof}

Now we can state and prove our main result.

\begin{theorem}
\label{t1} Let $X$ be a Banach lattice, $A$ and $B$ be positive operators on
$X$ such that $A \leq B$, $B$ is band irreducible, and $\rho(A) = \rho(B) >
0 $. Assume additionally that at least one of the following conditions is
satisfied.

\noindent $(a)$ The sequential Lorentz seminorm $l$ is a norm on $X$, $B$ is
$\sigma$-order continuous, and $\rho(B)$ is a pole of the resolvent $%
R(\lambda,B)$.

\noindent $(b)$ The sequential Lorentz seminorm $l$ is a norm on $X$, $B$ is
$\sigma$-order continuous, $\rho(A)$ is a pole of the resolvent $R(\lambda,
A)$, and the corresponding spectral subspace of $A$ is finite dimensional.

\noindent $(c)$ The operator $B$ is weakly compact, $\sigma$-order
continuous, and $\rho(B)$ is a pole of the resolvent $R(\lambda,B)$.

\noindent $(d)$ The operator $A$ is band irreducible, weakly compact, and $%
\rho(A)$ is a pole of the resolvent $R(\lambda,A)$. The operator $B$ is $%
\sigma$-order continuous.

Then $A = B$.
\end{theorem}

\begin{proof}
Part $(a)$. By part $(a)$ of Lemma~\ref{l3} we have $\rho(B) >0$ and the
corresponding eigenspace is one dimensional. Let $P$ be the corresponding
spectral projection. Without loss of generality we can assume that $\rho(B)
= 1$. For a positive $\varepsilon$ let $B_\varepsilon = \varepsilon A + (1 -
\varepsilon)B$. Notice that $A \leq B_\varepsilon \leq B$ and therefore $%
\rho(B_\varepsilon) =1$. Then (see~\cite[Lemma VII.6.8]{DS}) for a small
enough $\varepsilon$ the point $1$ will be an isolated eigenvalue of $%
B_\varepsilon$ of multiplicity $1$. Let $Q$ be the corresponding spectral
projection. The formula $(\bigstar)$ applied to the projection $Q$ provides $%
Q = \lim \limits_{\lambda \downarrow 1} (\lambda - 1) \sum
\limits_{n=0}^\infty \frac{(B_\varepsilon)^n}{\lambda^{n+1}}$ whence $Q \geq
0$ and therefore there is $g \in X$, $g \gvertneqq 0$, and $B_\varepsilon g=
g$.

Next we notice that $Bg \geq B_\varepsilon g = g$. Assume for a moment that $%
Bg \neq g$. The formula $(\bigstar)$ and the inequality $B \geq
B_\varepsilon $ guarantee that $P \geq Q$ whence $Pg \geq Qg =g$. Notice
that our assumption $Bg \neq g$ implies that $Pg \neq g$ because $BPg = Pg$.
Next we see that $B(Pg - g) = Pg - Bg \geq 0$ whence $Pg \geq Bg \geq g$.
From the last inequality we have $-g \geq -Bg$ and therefore $Pg - g \geq Pg
- Bg = B(Pg - g) \geq 0$. We see now that operator $B$ leaves invariant the
principal ideal generated by $Pg - g$ and because $B$ is $\sigma$-order
continuous it leaves invariant the nonzero principal band $\mathfrak{B}$
generated by $Pg -g$. But $B$ is band irreducible whence $\mathfrak{B} = X$.
Next notice that $P(Pg - g) = 0$ and that the operator $P$ is positive and $%
\sigma$-order continuous by Lemma~\ref{l3} whence $P=0$, a contradiction.

We have just proved that $Bg = B_\varepsilon g = g$. The operator $B -
B_\varepsilon$ is $\sigma$-order continuous because $0 \leq B -
B_\varepsilon \leq B$ and therefore $B - B_\varepsilon$ is zero on the
principal band $\mathfrak{C} = \{g\}^{dd}$ generated by $g$. But it follows
immediately from $Bg = g$ and the fact that $B$ is $\sigma$-order continuous
that $B\mathfrak{C} \subseteq \mathfrak{C}$. Recalling that $B$ is band
irreducible we see that $\mathfrak{C} = X$ whence $B_\varepsilon = B$ and
therefore $A = B$.

Part $(b)$. Assume again that $\rho(A) = \rho(B) = 1$. For a positive $%
\varepsilon$ consider the operator $B_\varepsilon = (1 - \varepsilon)A +
\varepsilon B$. Then $A \leq B_\varepsilon \leq B$ whence $%
\rho(B_\varepsilon) = 1$. Lemma VII.6.8 in ~\cite{DS} guarantees that for a
small enough $\varepsilon$, $1$ is a pole of the resolvent $R(\lambda,
B_\varepsilon)$. Next notice that the operator $B_\varepsilon$ is $\sigma$%
-order continuous because $B_\varepsilon \leq B$. We claim that $%
B_\varepsilon$ is band irreducible. Indeed, let $\mathfrak{B}$ be a
nontrivial $B_\varepsilon$ invariant band in $X$. Then $A \mathfrak{B}
\subseteq \mathfrak{B}$ because $A \leq B_\varepsilon$. But $B$ is a linear
combination of $A$ and $B_\varepsilon$ whence $B\mathfrak{B} \subseteq
\mathfrak{B}$ in contradiction to our assumption that $B$ is band
irreducible.

By part $(a)$ we have $A = B_\varepsilon$ whence $A= B$.

Part $(c)$. We can assume that $\rho(B) = 1$. Our assumptions about the
operator $B$ and part $(b)$ of Lemma~\ref{l3} guarantee that $1$ is an
eigenvalue of $B$ of multiplicity one and that the corresponding
one-dimensional spectral projection $P$ is $\sigma$-order continuous. The
rest of the proof is the same as in part $(a)$.

Part $(d)$. Assume $\rho (A)=\rho (B)=1$. As in part $(c)$ our assumptions
guarantee that $1$ is the pole of $\rho (\lambda ,A)$ of multiplicity one
and that the corresponding one-dimensional spectral projection $P_{A}$ is $%
\sigma $-order continuous. Let $g$ be the corresponding positive
eigenvector. Substituting, if needed, $B$ by $B_{\varepsilon
}=(1-\varepsilon )A+\varepsilon B$ where $\varepsilon $ is a small enough
positive number we can assume that $1$ is the pole of $\rho (\lambda ,B)$
also of multiplicity one. Let $P_{B}$ be the corresponding one-dimensional
spectral projection. Notice that $Bg\geq Ag=g$. Assume that $Bg\gneqq g$.
Then exactly as in part (a), we conclude that $P_{B}g\geq Bg\geq g\geq 0$
and that $P_{B}g-g\geq B(P_{B}g-g)\geq A(P_{B}g-g)\geq 0$. Since both $A$
and $B$ are $\sigma $-order continuous they both leave invariant the
principal band $\mathfrak{B}=\{P_{B}g-g\}^{dd}$. But $A$ and $B$ are band
irreducible. Therefore $\mathfrak{B}=X$. However, since $P_{B}\geq P_{A}\geq
0$ and $P_{B}(P_{B}g-g)=0$, we have $P_{A}(P_{B}g-g)=0$. We already noted
that Lemma 0.8 implied that $P_{A}$ is $\sigma $-order continuous. Therefore
$P_{A}(\mathfrak{B})=P_{A}(X)=\{0\}$. This is a contradiction since $P_{A}$
is non-zero.
\end{proof}

\begin{remark}
\label{r5} We can further weaken the conditions in parts $(c)$ and $(d)$ of
Theorem~\ref{t1} (Compare with Theorem 6 in~\cite{Ki} and with ~\cite{Ki1}).
Namely, in part $(c)$ instead of assuming that $B$ is $\sigma $-order
continuous and weakly compact we can assume that there are positive
operators $R$ and $S$ such that $R$ is weakly compact, $S\leq R$ and $S\neq 0
$, $SB\leq BS$, and either $R$ is order continuous, or $R$ is $\sigma $%
-order continuous and $S$ is band irreducible (in connection with this see~%
\cite{Ki1}). Similar changes can be made for part $(d)$. (See the appendix.)
\end{remark}

\noindent We will now discuss in more detail the conditions in the statement
of Theorem~\ref{t1}.

Concerning parts(a) and (b), we showed already in Example~\ref{e1} that the
condition that $B$ is $\sigma$-order continuous cannot be omitted. But to
show that we cannot dispense with the condition that the sequential Lorentz
seminorm is a norm on $X$ requires a more subtle example. For this purpose
we will use and modify the example from ~\cite[Proposition 13.5]{AAK} (See
also ~\cite[Section 3]{KW}).

\begin{example}
\label{e2} There are a Banach lattice $X$ (on which the sequential Lorentz
seminorm is not a norm) and positive operators $A$ and $B$ on $X$ such that

\begin{itemize}
\item $0 \leq A \leq B$.

\item $\rho(A)=\rho(B) =1$.

\item $\rho(B)$ is a pole of the resolvent $R(\lambda,B)$ of multiplicity
one.

\item $B$ is order continuous.

\item $B$ is band irreducible.
\end{itemize}

However, $A \neq B$.
\end{example}

\begin{proof}
Part 1. In this part we will construct the Banach space $X$ and the operator
$A$. Let $\psi_0$ denote a function on $[0,1]$ defined as follows:
\begin{equation*}
\psi_0(t) =
\begin{cases}
2t & \text{if $t \in [0,1/2]$,} \\
2 - 2t & \text{if $t \in [1/2,1]$},%
\end{cases}%
\end{equation*}
Using this function $\psi_0$, we define a mapping $\psi: [0,1] \to [0,1]$ by
letting $\psi(0) = 0$ and letting $\psi(t) = 2^{-n}\psi_0(2^nt-1)$ for $t
\in [2^{-n}, 2^{-n+1}]$ and $n \in \mathds{N}$. The graph of $\psi$ is shown
below.

\medskip \centerline{\epsffile{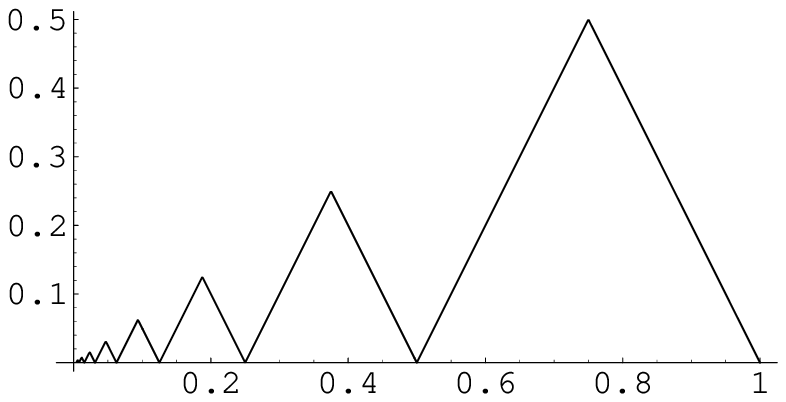}} \medskip Next we define the
mapping $\varphi : [0,2] \to [0,2]$ by letting
\begin{equation*}
\varphi(t) =
\begin{cases}
\psi(t) & \text{if $t \in [0,1]$,} \\
\psi(2 - t) & \text{if $t \in [1,2]$},%
\end{cases}%
\end{equation*}
We introduce the composition operator $T_\varphi$ on $C[0,2]$ by the formula
$T_\varphi f = f\circ \varphi, \; f \in C[0,2]$. We will describe
recursively the sets $\varphi^{-n}(0), \; n \in \mathds{N}$. Clearly $%
\varphi^{-1}(0) = \{0\} \cup \{1/2^k, \; k \in \mathds{N} \} \cup \{1\} \cup
\{2 -1/2^k, k \in \mathds{N}\} \cup \{2\}$. Let $n \in \mathds{N}$ and let $%
(a,b)$ be an interval complementary to $\varphi^{-n}(0)$ (i.e. $a,b \in
\varphi^{-n}(0)$ and $(a,b) \cap \varphi^{-n}(0) = \emptyset$). Then
\begin{equation*}
(a,b) \cap \varphi^{-(n+1)}(0) = \{\frac{(2^k - 1)a +b}{2^k} : k \in %
\mathds{N} \} \cup \{\frac{a+(2^k -1) b}{2^k} : k \in \mathds{N} \}%
\eqno{(\star \star \star)}
\end{equation*}%
.

From the description above it is easy to see that the mapping $\varphi$ has
the following properties:

\begin{enumerate}
\item For any $n \in \mathds{N}$, $0$ is the only fixed point of the mapping
$\varphi^n$ whence $\varphi^{-n}(0) \subset \varphi^{-(n+1)}(0), \; n \in %
\mathds{N}$.

\item For any $n \in \mathds{N}$, the set $\varphi^{-n}(0)$ is countable.

\item The set $\bigcup \limits_{n=1}^\infty \varphi^{-n}(0)$ is dense in $%
[0,2]$.
\end{enumerate}

We introduce a lattice seminorm $|\| \cdot \||$ on $C[0,2]$ as follows:
\begin{equation*}
|\| f \|| = \sup \limits_{n \in \mathds{N}} \{ \frac{1}{n!} \max \limits_{t
\in \varphi^{-n}(0)}|f(t)| \} \eqno{(\star \star)}
\end{equation*}
The condition $(3)$ above guarantees that this lattice seminorm is actually
a norm on $C[0,2]$. Let $X$ be the completion of $C[0,2]$ in this norm. Then
$X$ is a Banach lattice. Notice that the Lorentz seminorm on $X$ is not a
norm (see Lemma~\ref{l4} below). The obvious inequality $|\| T_\varphi f \||
\leq |\| f \||, f \in C[0,2]$ shows that $T_\varphi$ can be extended as a
continuous operator on $X$. We will denote the extension by $A$. Thus $A$ is
a positive operator on $X$ and $\|A\| \leq 1$. Moreover, the operator $A$
preserves disjointness and therefore it is a lattice homomorphism.

Part 2. We will prove here that the spectrum of $A$ on $X$ is the set $%
\{0,1\}$ and that 1 is an eigenvalue of $A$ of multiplicity 1. Indeed, the
space $X$ is the direct sum of the closed ideal $J$, $J = \{x \in X : x(0) =
0\}$ and the one-dimensional subspace $Y$ generated by the function $%
\mathds{1}$. Notice that $AY = Y$ and $AJ \subset J$. Let $x \in J$ and $k
\in \mathds{N}$, then
\begin{equation*}
|\|A^k x\|| = \sup \limits_{n \in \mathds{N}} \{ \frac{1}{n!} \max
\limits_{t \in \varphi^{-n}(0)}|x(\varphi^k(t))| \}.
\end{equation*}
But because $x \in J$ we have $A^k x \equiv 0$ on $\varphi^{-k}(0)$ and
therefore
\begin{equation*}
|\|A^k x\|| = \sup \limits_{n \geq k+1} \{ \frac{1}{n!} \max \limits_{t \in
\varphi^{-n}(0)}|x(\varphi^k(t))| \} \leq
\end{equation*}
\begin{equation*}
\sup \limits_{n \geq k+1} \{ \frac{1}{n!} \max \limits_{t \in
\varphi^{-n+k}(0)}|x(t)| \} \leq \sup \limits_{n \geq k+1} \{ \frac{1}{%
k!(n-k)!} \max \limits_{t \in \varphi^{-n+k}(0)}|x(t)| \} \leq
\end{equation*}
\begin{equation*}
\frac{1}{k!} |\|x\||.
\end{equation*}
Hence, the restriction $A|J$ is a quasinilpotent operator.

Part 3. We claim that the operator $A$ is order continuous. Let us consider $%
E = \bigcup \limits_{n=1}^\infty \varphi^{-n}(0) $ as a topological space
with the topology of inductive limit. Then every $x \in X$ can be
represented as a continuous function on $E$ and the principal ideal (not
closed!) $I$ in $X$ generated by $\mathds{1}$ can be identified with the
space of all bounded continuous functions on $E$. By the Kreins'-Kakutani
theorem the ideal $I$ is order isomorphic to $C(K)$ for some compact
Hausdorff space $K$. We see that $K$ is homeomorphic to the Stone - \v{C}ech
compactification $\beta E$ of $E$. \footnote{%
'In order to keep the flow of the proof we delay the verification of this
statement and some others to Lemma~\ref{l5}.} Next notice that $AI \subset I$
and that the operator $A$ induces a homomorphism of the algebra $C(\beta E)$
to which corresponds a continuous mapping $\tau$ of $\beta E$ into itself.
It is easy to see that $\tau$ is the unique extension of the map $\varphi :
E \to E$ on $\beta E$ (see Lemma~\ref{l5}).

It follows from the definition of $\varphi$ that if $U$ is a subset of $%
[0,2] $ and $Int U \neq \emptyset$ then $Int \varphi(U) \neq \emptyset$. The
definition of the inductive limit topology on $E$ shows that the map $%
\varphi : E \to E$ has the same property (see the proof of Lemma~\ref{l5}).
Recall that $E$ is a completely regular topological space and as such is
dense in $\beta E$. Therefore $\tau$ being the extension of $\varphi$ also
maps a set with with nonempty interior to a set with nonempty interior
(Lemma~\ref{l5}). Hence by Lemma~\ref{l1} below, $A|I$, the restriction of $%
A $ to the ideal generated by $\mathds{1}$ in $X$, is order continuous. From
this it follows easily that $A$ is order continuous. Indeed, assume to the
contrary that there exist a net $x_\alpha \downarrow 0$ and a $z \in X$, $z
\gvertneqq 0$ such that $Ax_\alpha \geq z$. But then $x_\alpha \wedge %
\mathds{1} \downarrow 0$ and $A(x_\alpha \wedge \mathds{1}) = Ax_\alpha
\wedge A\mathds{1} \geq z \wedge A\mathds{1} = z \wedge \mathds{1}
\gvertneqq 0$ (we have used here that $A$ is lattice homomorphism, $A%
\mathds{1} = \mathds{1}$, and $\mathds{1}$ is a weak unit in $X$ ) in
contradiction to $A|I$ being order continuous.

Part 4. Here we will construct the operator $B$. We need first to construct
a sequence of auxiliary operators on $X$. To do it we will need the
description of the sets $\varphi^{-n}(0), \; n \in \mathds{N}$, from Part 1.
The open subintervals of $[0,2]$ complementary to $\varphi^{-1}(0)$ and
ordered according to their natural order on the real line make an infinite
sequence $\{I_n\}_{n \in \mathds{Z}}$ where $\mathds{Z}$ as usual is the set
of all integers and $I_0 = (1/2,1)$. We will call these intervals -
intervals of the first order. Let us define a homeomorphism $\theta_1$ of $%
[0,2]$ onto itself as follows: $\theta_1$ maps the closure of every interval
$I_n$ onto the closure of its neighbor to the right, $I_{n+1}$, $n \in %
\mathds{Z}$, $\theta_1$ is linear on the closure of each $I_n$, $\theta_1(0)
= 0$ and $\theta_1(2) = 2$.

Now let us consider the set $\varphi^{-2}(0)$. It contains of course all the
points from $\varphi^{-1}(0)$ but also inside every interval of the first
order it has a countable number of new points which accumulate only to the
ends of that interval (see $(\star \star \star))$. Thus inside each interval
of the first order we have a countable sequence of intervals which we call
intervals of the second order. Next we construct the homeomorphism $\theta_2$
of $[0,2]$ onto itself with the following properties

\begin{itemize}
\item $\theta_2$ leaves the closure any interval of the first order
invariant.

\item $\theta_2$ maps the closure of every interval of the second order onto
the closure of its neighbor to the right.

\item $\theta_2$ is linear on the closure of each interval of the second
order.

\item $\theta_2(0) =0$ and $\theta_2(2) = 2$.
\end{itemize}

Continuing this way we construct a sequence $\{\theta_n\}$ of homeomorphisms
of $[0,2]$. The description of the sets $\varphi^{-n}(0)$ and the fact that
the homeomorphism $\theta_k$ leaves the set $\varphi^{-(k-1)}(0)$ invariant
and maps \textit{linearly} every complementary interval of $\varphi^{-k}(0)$
onto its neighbor to the right allow us to conclude that $\theta_i
(\varphi^{-n}(0)) = \varphi^{-n}(0), i,n \in \mathds{N}$. Thus (recall the
definition $(\star \star)$ of the norm on $X$) for any $i \in \mathds{N}$
the composition operator $T_i : T_ix(t) = x(\theta_i (t)), x \in X, t \in
\bigcup \limits_{n=1}^\infty \varphi^{-n}(0)$ is a positive disjointness
preserving isometry of $X$ onto itself. This isometry is obviously
invertible; the inverse isometry is also positive. Therefore $T_i$ is a
lattice isomorphism of $X$ onto itself and as such it is order continuous.

Let
\begin{equation*}
T = \sum \limits_{n, j_1, \ldots j_n \in \mathds{N} , i_1, \ldots , i_n, \in %
\mathds{Z}} \varepsilon_{i_1, \ldots , i_n , j_1, \ldots j_n} T^{i_1}_{j_1}
\ldots T^{i_n}_{j_n}
\end{equation*}
where the strictly positive numbers $\varepsilon_{i_1, \ldots , i_n , j_1,
\ldots j_n}$ are chosen so small that the series above converges in operator
norm. Because all the terms of the series are positive order continuous
operators the operator $T$ is also positive and order continuous (see~\cite[%
Proposition 2.13]{KW}).

We claim that the operator $T$ is band irreducible. Indeed, assume to the
contrary that there is a nontrivial $T$-invariant band $\mathfrak{B}$ in $X$%
. Recalling that if $x \in X$ and $x \equiv 0$ on $\bigcup
\limits_{n=1}^\infty \varphi^{-n}(0)$ then $x = 0$ we see that there are an $%
N \in \mathds{N}$ and two points $u_0, t \in \varphi^{-N}(0)$ such that

\begin{itemize}
\item $u_0$ and $t$ are isolated points of $\varphi^{-N}(0)$.

\item $x(u_0) =0$ for any $x \in \mathfrak{B}$.

\item There is a $z \in \mathfrak{B}$ such that $z(t) \neq 0$.
\end{itemize}

The construction of the mappings $\theta_i, \; i \in \mathds{N}$, shows that
there is an integer $i_1 \in \mathds{Z}$ such that the points $u_1
=\theta_1^{i_1}(u_0)$ and $t$ belong to the same subinterval of $[0,2]$ of
the first order. Next we can find another integer $i_2 \in \mathds{Z}$ such
that the points $u_2 = \theta_2^{i_2}(u_1)$ and $t$ are in the same interval
of order 2. Continuing this way we will the integers $i_1, \ldots , i_{N-1}$
such that the points $u_{N-1} = \theta_{N-1}^{i_{N-1}} \circ \ldots \circ
\theta_1^{i_1}(u_0)$ and $t$ are in the same interval $J$ of order $N-1$
and, being isolated points of $\varphi^{-N}(0)$ they are endpoints of some
subintervals of $J$ of order $N$. The mapping $\theta_N$ shifts these
intervals to the right and thus we can find the integer $i_N \in \mathds{Z}$
such that $\theta_N^{i_N}(u_{N-1}) = t$.

Let $S =T_1^{i_1} \ldots T_N^{i_N}$. Then $Sz(u_0) = z(t) \neq 0$ whence $S
\mathfrak{B} \nsubseteqq \mathfrak{B}$. On the other hand the definition of
the operator $T$ shows that there is a positive constant $C$ such that $S
\leq CT$ and therefore the band $\mathfrak{B}$ must be $S$-invariant, which
yields a contradiction.

Let $l$ be the function $l(t) =t, t \in [0,2]$ and $L$ be the corresponding
multiplication operator on $X$. Fix a positive number $\varepsilon$ and
consider $B = A + \varepsilon LT$. Then $B$ is order continuous band
irreducible positive operator on $X$. If $\varepsilon$ is small enough then~%
\cite[IV.3.5]{Ka} the disk $\{z \in \mathds{C} : |z-1| \leq 1/2\}$ contains
only one point of the spectrum $\sigma(B)$ and that point is an eigenvalue
of $B$ of multiplicity one. Let $\delta_0(x) =x(0), x \in X$. Then $B^\prime
\delta_0 = \delta_0$ whence $\rho(B) =1$ and we are done.
\end{proof}

The following three lemmas were used in the construction of Example~\ref{e2}%
. The last of them is most probably known but we were not able to find it in
the literature so we provide its proof for the sake of completeness.

\begin{lemma}
\label{l4} Let $X$ be the Banach lattice constructed in the proof of Example~%
\ref{e2}. Then the sequential Lorentz seminorm $l$ on $X$ is not a norm.
\end{lemma}

\begin{proof}
Fix $n \in \mathds{N}$. The set $\varphi^{-n}(0)$ is a closed countable
subset of $[0,2]$ and therefore we can find a sequence $\{f_{n,k}\}_{k \in %
\mathds{N}}$ of non-negative functions in $C[0,2]$ such that $f_{n,k} \equiv
0 $ on $\varphi^{-n}(0), \; k \in \mathds{N}$ and $f_{n,k} \mathop \uparrow
\limits_{k \to \infty} \mathds{1}$. But the definition of norm on $X$ (see $%
(\star \star)$) shows that $|\|f_{n,k}\|| \leq 1/(n+1)!, \; k \in \mathds{N}$%
, and according to the definition of the Lorentz seminorm $l$ (see $(\star)$%
) we have $l(\mathds{1}) = 0$.
\end{proof}

\begin{lemma}
\label{l5} Let $X$ be the Banach lattice constructed in the proof of Example~%
\ref{e2} and $E = \bigcup \limits_{n=1}^\infty \varphi^{-n}(0)$ be the
topological space endowed with the topology of inductive limit of the
compact spaces $\varphi^{-n}(0)$ (considered as closed subsets of $[0,2]$).
Then the principal ideal $I$ in $X$ generated by the function $\mathds{1}$, $%
I = \{x \in X : \; \exists \; c > 0\, \text{such that}\; |x| \leq c %
\mathds{1}\}$, is isometrically and lattice isomorphic to the space $C(\beta
E)$ of all continuous functions on the Stone - \u{C}ech compactification of $%
E$. Moreover, the mapping $\varphi : E \to E$ allows the unique continuous
extension $\tau : \beta E \to \beta E$ and if $H$ is a closed subset of $%
\beta E$ such that $Int \, H \neq \emptyset$ then $Int \, \tau(H) \neq
\emptyset$.
\end{lemma}

\begin{proof}
The definition $(\star \star )$ of the norm on $X$ shows that every $x\in X$
is a function on $E$, completely defined by its values on $E$ and continuous
on every set $\varphi ^{-n}(0)$, $n\in \mathds{N}$. The topology of
inductive limit on $E$ is the finest (strongest) topology on $E$ such that
every function continuous on $\varphi ^{-n}(0),\;n\in \mathds{N}$, is
continuous on $E$ (see e.g. ~\cite[page 44]{Ha}). Therefore $X$ is a proper
subspace of $C(E)$. We claim that if $f$ is a bounded continuous function on
$E$ then $f\in X$. Indeed, for every $n\in \mathds{N}$ we can find a
function $f_{n}\in C[0,2]$ such that $f_{n}\equiv f$ on $\varphi ^{-n}(0)$
and $\Vert f_{n}\Vert _{C[0,2]}\leq \sup\limits_{t\in E}|f(t)|$. It is
obvious from $(\star \star )$ that $\{f_{n}\}$ is a Cauchy sequence in $X$
and therefore converges to some $x\in X$; but $f\equiv x$ on $E$ whence $f=x$%
. Now we see that the principal ideal $I$ generated in $X$ by the function $%
\mathds{1}$ coincides with the algebra of all bounded continuous functions
on $E$ which can be identified with the space $C(\beta E)$.

To see this, it is sufficient to note that $E$ is completely regular. In
fact, let $F$ be a closed subset of $E$ and $t\in E\smallsetminus F$. For
some $n$, $t\in\varphi^{-n}(0)$. Let $f_{n}:\varphi^{-n}(0)\rightarrow%
\lbrack0,1]$ be a continuous Urysohn function such that $f_{n}(t)=1$ and $%
f_{n}(\varphi ^{-n}(0)\cap F)=\{0\}$. Let
\begin{equation*}
\widetilde{f_{n}}:\varphi^{-n}(0)\cup(\varphi^{-(n+1)}(0)\cap F)\rightarrow
\lbrack0,1]
\end{equation*}
be defined by%
\begin{equation*}
\widetilde{f_{n}}(s)=f_{n}(s)\text{, }s\in\varphi^{-n}(0)\text{ and}%
\widetilde{f_{n}}(s)=0\text{, }s\in(\varphi^{-(n+1)}(0)\cap
F)\smallsetminus\varphi^{-n}(0)\text{.}
\end{equation*}
It is straight forward to show that the function$\widetilde{f_{n}}$ defined
as above on the indicated closed subset of $\varphi^{-(n+1)}(0)$ is
continuous. Then by Tietze's extension theorem there is a continuous
function $f_{n+1}:\varphi^{-(n+1)}(0)\rightarrow\lbrack0,1]$ that extends$%
\widetilde {f_{n}}$. But this means it also extends $f_{n}$. Hence by
induction we construct a sequence of continuous functions $\{f_{n}\}$ such
that, for each $n$, $f_{n}:\varphi^{-n}(0)\rightarrow\lbrack0,1]$ with $%
f_{n}(t)=1$ and $f_{n}(\varphi^{-n}(0)\cap F)=\{0\}$ and $f_{n+1}$ is an
extension of $f_{n}$. Now we define a function $f:E\rightarrow\lbrack0,1]$
by $f(s)=f_{n}(s)$ for each $s\in E$ whenever $s\in\varphi^{-n}(0)$ for some
$n$. Clearly $f$ is well defined and satisfies $f(t)=1$ and $f(F)=\{0\}$.
Furthermore $f$ is continuous since $f_{|\varphi^{-n}(0)}=f_{n}$ is
continuous for each $n$. Whence $E$ is completely regular.

Returning to $C(\beta E)$ we notice that the operator $T_\varphi$ is a
unital homomorphism of this algebra and therefore in generates a continuous
map $\tau : \beta E \to \beta E$ such that the restriction of $\tau$ on $E$
coincides with $\varphi$. Let $V$ be an open nonempty subset of $E$. The
definition of inductive (or direct) limit topology (see again~\cite[page 44]%
{Ha}) shows that for any $n \in \mathds{N}$ the set $V \cap \varphi^{-n}(0)$
is an open subset of $\varphi^{-n}(0)$. Let $N \in \mathds{N}$ be the
smallest natural number such that $V \cap \varphi^{-N}(0) \neq \emptyset$.
Consider the subset $F$ of $E$ that consists of all the points from $%
\varphi^{-1}(0)$ and of all the points of the form $\frac{a+b}{2}$ where $a,
b \in \varphi^{-1}(0)$ and $(a,b) \cap \varphi^{(-1)}(0) = \emptyset$. We
can easily see that $F \subset \varphi^{(-2)}(0)$ is a closed subset of $%
[0,2]$ and of $E$ whence $W = V \setminus F$ is open in $E$.

Notice that it follows from $(\star \star \star)$ that every point from $%
\varphi^{-N}(0)$ is limit of a sequence of points from $\varphi^{-(N+1)}(0)
\setminus \varphi^{-N}(0)$. Therefore $\varphi^{-M}(0) \cap W \neq \emptyset$
where $M = \max{(3, N +1)}$. Let us now fix $n \in \mathds{N}$ such that $%
W_n = W \cap \varphi^{-n}(0) \neq \emptyset$. Let $G$ be an open subset of $%
[0,2] \setminus F$ such that $W_n = G \cap \varphi^{-n}(0)$. Then $G$ is the
union of disjoint open intervals $(c_j, d_j), \; j \in \mathds{N}$, in $%
[0,2] \setminus F$. Since $(c_j, d_j) \cap F = \emptyset$, the construction
of the map $\varphi$ shows that $\varphi((c_j,d_j))$ is an open subinterval
of $[0,1/2]$ (see the graph in part 1 of Example~\ref{e2}). Thus $%
\varphi(W_n) = \big{(} \bigcup \limits_{j=1}^\infty \varphi((c_j,d_j)) %
\big{)} \cap \varphi^{-(n-1)}(0)$ is an open nonempty subset of $%
\varphi^{-(n-1)}(0)$.

We have just proved that the set $\varphi(W)$ is non-empty and open in $E$
whence $Int \, \varphi(V) \neq \emptyset$. It remains to see that $\tau$
inherits the same property.

Let $U$ be a non-empty open subset of $\beta E$ and let $R=U\cap E$. Then $%
cl_{E}R$ has non-empty interior. Therefore $\varphi (cl_{E}R)$ also has
non-empty interior. That is there exists a non-empty open set $S$ in $\beta E
$ such that $S\cap E\subset int\varphi (cl_{E}R)$. But $cl_{\beta E}(S\cap
E)=cl_{\beta E}S$. Now notice that $cl_{E}(S\cap E)\subset \varphi
(cl_{E}R)\subset \tau (cl_{\beta E}U)$. Therefore $S\subset cl_{\beta
E}S=cl_{\beta E}(S\cap E)\subset \tau (cl_{\beta E}U)$ and $int\tau
(cl_{\beta E}U)\neq \emptyset $.
\end{proof}

\begin{lemma}
\label{l1} Let $K$ be a compact Hausdorff space, $\varphi : K \to K$ be a
continuous mapping, and $T_\varphi$ be the corresponding composition
operator on $C(K)$. The following conditions are equivalent.

\begin{enumerate}
\item The operator $T_\varphi$ is order continuous.

\item If $U$ is a closed subset of $K$ and $Int \,U \neq \emptyset$ then $%
Int \, \varphi(U) \neq \emptyset$.
\end{enumerate}
\end{lemma}

\begin{proof}
$(1) \Rightarrow (2)$ Assume to the contrary that there is a closed subset $%
U $ of $K$ such that $Int \, U \neq \emptyset$ but $Int \, \varphi(U) =
\emptyset$. Then we can find a net $\{f_\alpha\}$ in $C(K)$ such that $%
f_\alpha \downarrow 0$ and $f_\alpha = 1$ on $\varphi(U)$. Let $g \in C(K)$
be such that $g \neq 0$, $supp \; g \subset Int \, U$, and $0 \leq g \leq 1$%
. But then $T_\varphi f_\alpha \geq g$ in contradiction to $T_\varphi$ being
order continuous.

$(2) \Rightarrow (1)$ Assume to the contrary that the condition $(2)$ is
satisfied but $T_\varphi$ is not order continuous. Let $\{f_\alpha\}$ be a
net in $C(K)$ such that $f_\alpha \downarrow 0$ and $T_\varphi f_\alpha \geq
g \gvertneqq 0$ where $g \in C(K)$. If we fix a small enough positive $%
\varepsilon$ then the set $U = \{k \in K : g(k) \geq \varepsilon \}$ is a
closed subset of $K$ and $Int \, U \neq \emptyset$. But then $f_\alpha \geq
\varepsilon$ on $\varphi(U)$ and because $Int \, \varphi(U) \neq \emptyset$
we have a contradiction to $f_\alpha \downarrow 0$.
\end{proof}

\bigskip

\begin{remark}
\label{r3} Consider the operator $A$ constructed in Example~\ref{e2}. The
powers $A^n, n \in \mathds{N}$, are order continuous positive operators and
it can be easily seen that they converge in the operator norm to the
spectral projection on the one-dimensional subspace generated by the
eigenvector $\mathds{1}$. This projection is not order continuous because
its kernel is an ideal but not a band in $X$ (it is not even $\sigma$-order
continuous because the Banach lattice $X$ is separable and every band in it
is a principal band). It was proved in~\cite[Theorem 2.16]{KW} that if we
have a Banach lattice $Z$ and the Lorentz seminorm is a norm on $Z$ then the
operator norm limit of a sequence of order continuous positive operators is
also order continuous. Therefore we see that the condition that the Lorentz
seminorm is a norm cannot be omitted from the cited above theorem. The
corresponding example in~\cite{KW} (see~\cite[Section 3]{KW}) contains
typing errors in its definition. It can be corrected to yield the desired
result. However the proof given there needs to be corrected as well.
\end{remark}

Now we will continue to discuss the conditions in the statement of Theorem~%
\ref{t1} part (c). As the following example shows it is not possible to omit
the condition that $B$ is order continuous.

\begin{example}
\label{e4} We will first reproduce Example 5.1 from~\cite{AAB}. Let $X =
M_\psi$ be a Marcinkiewicz space on $(0,1)$ generated by the function $\psi$
such that $\lim \limits_{t \to 0} \frac{\psi(2t)}{\psi(t)} = 1$. Then (see~%
\cite{Lo}) there is a singular positive functional $f$ on $X$ such that $f$
does not annihilate any nonzero band in $X$. Let $e$ be a positive weak unit
in $X$ such that $f(e) = 0$. Then it is easy to see that the one dimensional
operator $T$ on $X$ defined as $Tx = f(x)e, x \in X$ is band irreducible and
$T^2 = 0$.

Now we can use the example above to construct the desired counterexample.
Let $g$ be a positive functional on $X$ such that $g(e) = 1$, let $Ax
=g(x)e, x \in X$ and let, $B = A + T$. Then clearly $\sigma(A) = \{0,1\}$
whence $\rho(B) \geq 1$. But $\rho(B)$ is an eigenvalue of $B$ and it is
immediate to see that $e$ is the only eigenvector of $B$ corresponding to a
non-zero eigenvalue, and that the corresponding eigenvalue is $1$.
\end{example}

\begin{remark}
\label{r2} The famous theorem of de Pagter ~\cite{Pa} states that any ideal
irreducible positive compact operator has nonzero spectral radius. Thus
Example~\ref{e4} illustrates the difference between compact band irreducible
and compact ideal irreducible operators. In connection with this example the
following problem might be of independent interest.
\end{remark}

\begin{problem}
\label{pr1} On which Banach lattices do there exist positive compact band
irreducible quasinilpotent operators?
\end{problem}

\centerline{ \textbf{Appendix}}

\begin{proposition}
\label{prop1} The condition in part (b) of Lemma~\ref{l3} that the operator $%
T$ is weakly compact can be substituted by the condition that $T$ commutes
with a positive weakly compact nonzero operator $S$ and either $S$ is order
continuous or $S$ is $\sigma $-order continuous and band irreducible.
Respectively we can change the conditions in parts $(c)$ and $(d)$ of
Theorem~\ref{t1}.
\end{proposition}

\begin{proof}
As it was noticed in the proof of Lemma~\ref{l3} it is enough to prove that
the operator $Q_{-p} = \lim \limits_{\lambda \to 1} (\lambda
-1)^pR(\lambda,T)$ is $\sigma$-order continuous. Assume to the contrary that
$x_n \downarrow 0$ but $Q_{-p}x_n \geq y \gvertneqq 0$. We have $Sx_n %
\mathop \rightarrow \limits^w 0$ whence $Q_{-p}Sx_n \mathop \rightarrow
\limits^w 0$ and because $SQ_{-p} = Q_{-p}S$ $SQ_{-p}x_n \mathop \rightarrow
\limits^w 0$. But $SQ_{-p}x_n \geq Sy$ whence $Sy =0$.

Assume first that $S$ is order continuous. Let $Z$ be the maximal by
inclusion ideal in $X$ such that $SZ = 0$. Then $Z$ is a nonzero band in $X$%
. We claim that $TZ \subseteq Z$. Indeed, otherwise there are positive
nonzero $z \in Z$ and $u \perp Z$ such that $Tz \geq u$. But $STz = TSz = 0$
whence $Su = 0$ in contradiction with maximality of $Z$. But $T$ is band
irreducible whence $Z = X$ and $S = 0$, a contradiction.

Next assume that $S$ is $\sigma $-order continuous and band irreducible.
Then $S\{y\}^{dd}=0$ whence $\{y\}^{dd}=X$ and $S=0$, a contradiction again.
\end{proof}

\bigskip

\textbf{Acknowledgement.} We thank E. Alekhno who read the first version of
the paper (arXiv: 1205.5583) and made a number of valuable comments. Some of
these have been incorporated into the present paper. In addition, three
papers have recently appeared (see~\cite{Al1}, ~\cite{Ga}, ~\cite{BMR})
containing independently obtained results partially intersecting with our
Theorem~\ref{t1}. We thank, respectively, E. Alekhno, V. Troitsky and N.
Gao, and J. Bernik for alerting us to the respective papers.

\end{document}